\documentclass[11pt,twoside]{article}
\pdfoutput=1

\usepackage[margin=1in]{geometry}

\newcommand{\norm}[1]{\left\lVert#1\right\rVert}
\usepackage{mathrsfs}
\usepackage{bm}%

\usepackage{graphicx}
\usepackage{amsmath,amsfonts,amssymb}
\usepackage[mathscr]{eucal}
\usepackage{bm}

\usepackage{amsthm}
\usepackage{cite}
\usepackage{hyperref}

\usepackage[nameinlink]{cleveref}
\numberwithin{equation}{section}

\newtheorem{theorem}{Theorem}[section]
\newtheorem{lemma}[theorem]{Lemma}
\newtheorem{proposition}[theorem]{Proposition}

\theoremstyle{definition}

\theoremstyle{remark}

\usepackage{fancyhdr}
\begin{document}
%------------------------------------------------------------------------------------------------

\title{An EDG Method for Distributed Optimal Control of Elliptic PDEs}

	\author{Xiao Zhang%
		\thanks{College of Mathematics, Sichuan University,
			Chengdu, China (zhangxiaofem@163.com). X.~Zhang thanks Missouri University of Science and Technology for hosting him as a visiting scholar; some of this work was completed during his research visit.}
		\and
		Yangwen Zhang%
		\thanks{Department of Mathematics and Statistics, Missouri University of Science and Technology, Rolla, MO (\mbox{ywzfg4@mst.edu}, singlerj@mst.edu). Y.~Zhang and J.~Singler were supported in part by National Science Foundation grant DMS-1217122.  Y.~Zhang and J.~Singler thank the IMA for funding research visits, during which some of this work was completed.}
		\and
		John~R.~Singler%
		\footnotemark[2]
}

\date{}

\maketitle

\begin{abstract}
We consider a distributed optimal control problem governed by an  elliptic PDE, and propose an embedded discontinuous Galerkin (EDG) method to approximate the solution. We derive optimal a priori error estimates for the state, dual state, the optimal control, and suboptimal estimates for the fluxes. We present numerical experiments to confirm our theoretical results. 
\end{abstract}

%------------------------------------------------------------------------------------------------
\section{Introduction}
\label{intro}
%------------------------------------------------------------------------------------------------

We consider approximating the solution of the following distributed control problem. Let $\Omega\subset \mathbb{R}^{d} $ $ (d\geq 2)$ be a Lipschitz polyhedral domain  with boundary $\Gamma = \partial \Omega$.  The goal is to minimize
\begin{align}
J(u)=\frac{1}{2}\| y- y_{d}\|^2_{L^{2}(\Omega)}+\frac{\gamma}{2}\|u\|^2_{L^{2}(\Omega)}, \quad \gamma>0, \label{cost1}
\end{align}
subject to
\begin{equation}\label{Ori_problem}
\begin{split}
-\Delta y&=f+u ~\quad\text{in}~\Omega,\\
y&=g\qquad\quad\text{on}~\partial\Omega,
\end{split}
\end{equation}
It is well known that the optimal control problem \eqref{cost1}-\eqref{Ori_problem} is equivalent to the optimality system
\begin{subequations}\label{eq_adeq}
	\begin{align}
	-\Delta y &=f+u\quad~\text{in}~\Omega,\label{eq_adeq_a}\\
	y&=g\qquad~~~~\text{on}~\partial\Omega,\label{eq_adeq_b}\\
	-\Delta z &=y_d-y\quad~\text{in}~\Omega,\label{eq_adeq_c}\\
	z&=0\qquad\quad~~\text{on}~\partial\Omega,\label{eq_adeq_d}\\
	z-\gamma  u&=0\qquad\quad~~\text{in}~\Omega.\label{eq_adeq_e}
	\end{align}
\end{subequations}

Different numerical methods for optimal control problems governed by partial differential equations have been extensively studied by many researchers.  Numerical methods that have been investigated for this kind of problem include approaches based on standard finite element methods \cite{MR2470142,MR2486088,MR3473693,MR1686151,MR2493560}, mixed finite elements \cite{MR2585589,MR3679859,MR3427830,MR3103238,MR2998296,MR2576747,MR2398768}, and discontinuous Galerkin (DG) methods \cite{MR3022208,MR2644299}.

Recently, hybridizable discontinuous Galerkin (HDG) methods have been developed for many partial differential equations; see, e.g., \cite{MR2485455,MR2772094,MR2513831,MR2558780,MR2796169,MR3626531,MR3522968,MR3463051,MR3452794,MR3343926}.  HDG methods keep the advantages of DG methods and mixed methods, while also having less globally coupled unknowns.  HDG methods have now also been applied to many different  optimal control problems \cite{MR3508834,HuShenSinglerZhangZheng_HDG_Dirichlet_control1,HuShenSinglerZhangZheng_HDG_Dirichlet_control2,HuShenSinglerZhangZheng_HDG_Dirichlet_control3}.

The embedded discontinuous Galerkin (EDG) methods, originally proposed in \cite{MR2317378}, are obtained from HDG methods by replacing the discontinuous finite element space for the numerical traces with a continuous space.  The number of degrees of freedoms for the EDG method are much smaller than the HDG method.  This gain in computational efficiency can come with a loss: for the Poisson equation, convergence rates for the EDG method are one order lower than the HDG method \cite{MR2551142}.  However, for problems with strong convection the enhanced convergence properties of HDG methods are reduced \cite{FuQiuZhang15}.  Therefore, EDG methods are competitive for such problems, and researchers have recently begun to thoroughly investigate EDG methods for various partial differential equations \cite{peraire2011embedded,MR3404541,fernandez2016,MR3528316,fu2017analysis}.

Our long term goal is to devise efficient and accurate methods for complicated optimal flow control problems.  EDG methods have potential for such problems; therefore, as a first step, we consider an EDG method to approximate the solution of the above optimal control problem for the Poisson equation.  We use an EDG method with polynomials of degree $k$ to approximate all the variables of the optimality system \eqref{eq_adeq}, i.e., the state $y$, dual state $z$, the numerical traces, and the fluxes $\bm q = -\nabla  y $ and $ \bm p = -\nabla z$.  We describe the method in \Cref{sec:EDG}, and in \Cref{sec:analysis} we obtain the error estimates
\begin{align*}
&\norm{y-{y}_h}_{0,\Omega} = O( h^{k+1}),  &  &\norm{z-{z}_h}_{0,\Omega} = O( h^{k+1}),\\
&\norm{\bm{q}-\bm{q}_h}_{0,\Omega} = O( h^{k}), &  &\norm{\bm{p}-\bm{p}_h}_{0,\Omega} = O( h^{k}),
\end{align*}
and
\begin{align*}
&\norm{u-{u}_h}_{0,\Omega} = O( h^{k+1}).
\end{align*}
We present numerical results in \Cref{sec:numerics}, and then briefly discuss future work.

%--------------------------------------------------------
\section{EDG scheme for the optimal control problem}
\label{sec:EDG}
%--------------------------------------------------------

\subsection{Notation}Throughout the paper we adopt the standard notation $W^{m,p}(\Omega)$ for Sobolev spaces on $\Omega$ with norm $\|\cdot\|_{m,p,\Omega}$ and seminorm $|\cdot|_{m,p,\Omega}$. We denote $W^{m,2}(\Omega)$ by $H^{m}(\Omega)$ with norm $\|\cdot\|_{m,\Omega}$ and seminorm $|\cdot|_{m,\Omega}$, and also $H_0^1(\Omega)=\{v\in H^1(\Omega):v=0 \;\mbox{on}\; \partial \Omega\}$.  We denote the $L^2$-inner products on $L^2(\Omega)$ and $L^2(\Gamma)$ by
\begin{align*}
(v,w) &= \int_{\Omega} vw  \quad \forall v,w\in L^2(\Omega),\\
\left\langle v,w\right\rangle &= \int_{\Gamma} vw  \quad\forall v,w\in L^2(\Gamma).
\end{align*}
Furthermore, $ H(\text{div},\Omega) = \{\bm{v}\in [L^2(\Omega)]^d, \nabla\cdot \bm{v}\in L^2(\Omega)\} $.

Let $\mathcal{T}_h$ be a collection of disjoint elements that partition $\Omega$.  We denote by $\partial \mathcal{T}_h$ the set $\{\partial K: K\in \mathcal{T}_h\}$. For an element $K$ of the collection  $\mathcal{T}_h$, let $e = \partial K \cap \Gamma$ denote the boundary face of $ K $ if the $d-1$ Lebesgue measure of $e$ is non-zero. For two elements $K^+$ and $K^-$ of the collection $\mathcal{T}_h$, let $e = \partial K^+ \cap \partial K^-$ denote the interior face between $K^+$ and $K^-$ if the $d-1$ Lebesgue measure of $e$ is non-zero. Let $\varepsilon_h^o$ and $\varepsilon_h^{\partial}$ denote the set of interior and boundary faces, respectively. We denote by $\varepsilon_h$ the union of  $\varepsilon_h^o$ and $\varepsilon_h^{\partial}$. We finally introduce
\begin{align*}
(w,v)_{\mathcal{T}_h} = \sum_{K\in\mathcal{T}_h} (w,v)_K,   \quad\quad\quad\quad\left\langle \zeta,\rho\right\rangle_{\partial\mathcal{T}_h} = \sum_{K\in\mathcal{T}_h} \left\langle \zeta,\rho\right\rangle_{\partial K}.
\end{align*}

Let $\mathcal{P}^k(D)$ denote the set of polynomials of degree at most $k$ on a domain $D$.  We introduce the discontinuous finite element spaces
\begin{align}
\bm{V}_h  &:= \{\bm{v}\in [L^2(\Omega)]^d: \bm{v}|_{K}\in [\mathcal{P}^k(K)]^d, \forall K\in \mathcal{T}_h\},\\
{W}_h  &:= \{{w}\in L^2(\Omega): {w}|_{K}\in \mathcal{P}^{k}(K), \forall K\in \mathcal{T}_h\},\\
{M}_h  &:= \{{\mu}\in L^2(\varepsilon_h): {\mu}|_{e}\in \mathcal{P}^k(e), \forall e\in \varepsilon_h\}.
\end{align}
Let  $M_h(o)$ and $M_h(\partial)$  denote the spaces defined in the same way as $M_h$, but with $ \varepsilon_h $ replaced by $ \varepsilon_h^o$ and $ \varepsilon_h^{\partial}$, respectively.  Spatial derivatives of functions in these discontinuous finite element spaces are understood to be taken piecewise on each element $K\in \mathcal T_h$.

For EDG methods, we replace the discontinuous finite element space $M_h$ for the numerical traces with the continuous finite element space $\widetilde{M}_h$ defined by
\begin{equation}
\widetilde{M}_h:=M_h \cap \mathcal{C}^0 (\varepsilon_h).
\end{equation}
The spaces $\widetilde{M}_h(o)$ and $\widetilde{M}_h(\partial)$ are defined in the same way as $M_h(o)$ and $M_h(\partial)$.

%--------------------------------------------------------
\subsection{The EDG Formulation}
%--------------------------------------------------------

The mixed weak form of the optimality system \eqref{eq_adeq_a}-\eqref{eq_adeq_e} is given by
\begin{subequations}\label{mixed}
	\begin{align}
	(\bm q,\bm r_1)-( y,\nabla\cdot \bm r_1)+\langle y,\bm r_1\cdot \bm n\rangle&=0,\label{mixed_a}\\
	(\nabla\cdot\bm q,  w_1)&= ( f+ u, w_1),  \label{mixed_b}\\
	(\bm p,\bm r_2)-(z,\nabla \cdot\bm r_2)+\langle z,\bm r_2\cdot\bm n\rangle&=0,\label{mixed_c}\\
	(\nabla\cdot\bm p,  w_2)&= (y_d- y, w_2),  \label{mixed_d}\\
	( z-\gamma u,v)&=0,\label{mixed_e}
	\end{align}
\end{subequations}
for all $(\bm r_1, w_1,\bm r_2, w_2,v)\in H(\text{div},\Omega)\times L^2(\Omega)\times H(\text{div},\Omega)\times L^2(\Omega)\times L^2(\Omega)$.  Note that the optimality condition \eqref{mixed_e} gives $ u = \gamma^{-1} z $.

The EDG method seeks approximate fluxes ${\bm{q}}_h,{\bm{p}}_h \in \bm{V}_h $, states $ y_h, z_h \in W_h $, interior element boundary traces $ \widehat{y}_h^o,\widehat{z}_h^o \in \widetilde{M}_h(o) $, and  control $ u_h \in W_h$ satisfying
\begin{subequations}\label{HDG_discrete2}
	\begin{align}
	%%%%%%%%%%%%%
	(\bm q_h,\bm r_1)_{\mathcal T_h}-( y_h,\nabla\cdot\bm r_1)_{\mathcal T_h}+\langle \widehat y_h^o,\bm r_1\cdot\bm n\rangle_{\partial\mathcal T_h\backslash \varepsilon_h^\partial}&=-\langle  I_hg,\bm r_1\cdot\bm n\rangle_{\varepsilon_h^\partial}, \label{HDG_discrete2_a}\\
	%%%%%%%%%%%%%%%%%%%%
	-(\bm q_h,  \nabla w_1)_{\mathcal T_h} +\langle\widehat {\bm q}_h\cdot\bm n,w_1\rangle_{\partial\mathcal T_h} 
	  -  ( u_h, w_1)_{\mathcal T_h} &=  ( f, w_1)_{\mathcal T_h},   \label{HDG_discrete2_b}%  \nonumber \\
	% 
%	& \quad  + ( f, w_1)_{\mathcal T_h}   \label{HDG_discrete2_b}
	\end{align}
	for all $(\bm{r}_1, w_1)\in \bm{V}_h\times W_h$, where $I_h g$ is a continuous interpolation of $g$ on $\varepsilon_h^\partial$, 
	\begin{align}
	%%%%%%%%%%%%%%%%%%%
	(\bm p_h,\bm r_2)_{\mathcal T_h}-(z_h,\nabla\cdot\bm r_2)_{\mathcal T_h}+\langle \widehat z_h^o,\bm r_2\cdot\bm n\rangle_{\partial\mathcal T_h\backslash\varepsilon_h^\partial}&=0,\label{HDG_discrete2_c}\\
	%%%%%%%%%%%%%%%%%%%%%
	-(\bm p_h, \nabla w_2)_{\mathcal T_h}+\langle\widehat{\bm p}_h\cdot\bm n,w_2\rangle_{\partial\mathcal T_h} +  ( y_h, w_2)_{\mathcal T_h}&= (y_d, w_2)_{\mathcal T_h},  \label{HDG_discrete2_d}
	\end{align}
	for all $(\bm{r}_2, w_2)\in \bm{V}_h\times W_h$.
	\begin{align}
	\langle\widehat {\bm q}_h\cdot\bm n,\mu_1\rangle_{\partial\mathcal T_h\backslash\varepsilon^{\partial}_h}&=0\label{HDG_discrete2_e},\\
	\langle\widehat{\bm p}_h\cdot\bm n,\mu_2\rangle_{\partial\mathcal T_h\backslash\varepsilon^{\partial}_h}&=0,\label{HDG_discrete2_f}
	\end{align}
%	for all $\mu_1\in M_h(o)$ 
%	\begin{align}
%	\langle\widehat{\bm p}_h\cdot\bm n-\bm \beta\cdot\bm n\widehat z_h^o,\mu_2\rangle_{\partial\mathcal T_h\backslash\varepsilon^{\partial}_h}&=0,\label{HDG_discrete2_f}
%	\end{align}
	for all $\mu_1,\mu_2\in \widetilde{M}_h(o)$, and the optimality condition 
	\begin{align}
	(z_h-\gamma u_h, w_3)_{\mathcal T_h} &= 0\label{HDG_discrete2_g},
	\end{align}
	for all $ w_3\in W_h$.  The EDG discrete optimality condition \eqref{HDG_discrete2_g} gives $ u_h = \gamma^{-1} z_h $.  The numerical traces on $\partial\mathcal{T}_h$ are defined by 
	\begin{align}
	\widehat{\bm{q}}_h\cdot \bm n &=\bm q_h\cdot\bm n+h^{-1} (y_h-\widehat y_h^o)   \quad \mbox{on} \; \partial \mathcal{T}_h\backslash\varepsilon_h^\partial, \label{HDG_discrete2_h}\\
	\widehat{\bm{q}}_h\cdot \bm n &=\bm q_h\cdot\bm n+h^{-1}  (y_h-I_hg)  ~ \   \mbox{on}\;  \varepsilon_h^\partial, \label{HDG_discrete2_i}\\
	\widehat{\bm{p}}_h\cdot \bm n &=\bm p_h\cdot\bm n+h^{-1} (z_h-\widehat z_h^o)\quad \mbox{on} \; \partial \mathcal{T}_h\backslash\varepsilon_h^\partial,\label{HDG_discrete2_j}\\
	\widehat{\bm{p}}_h\cdot \bm n &=\bm p_h\cdot\bm n+h^{-1}  z_h\quad\quad\quad\quad\mbox{on}\;  \varepsilon_h^\partial.\label{HDG_discrete2_k}
	\end{align}
\end{subequations}
%this completes the definition of the EDG method.
Our implementation of the above EDG method and the local solver is similar to the implementation of an HDG scheme for a similar problem described in detail in \cite{HuShenSinglerZhangZheng_HDG_Dirichlet_control2}.

\section{Error Analysis}
\label{sec:analysis}

Next, we provide a convergence analysis of the above EDG method for the optimal control problem.  Throughout this section, we assume $ \Omega $ is a bounded convex polyhedral domain, the problem data satisfies $ f \in L^2(\Omega) $ and $ g \in \mathcal{C}^0(\partial \Omega) $, $ h \leq 1 $, and the solution of the optimality system \eqref{eq_adeq} is sufficiently smooth.

Below, we prove our main convergence result:
\begin{theorem}\label{main_res}
	We have
	\begin{align*}
	\|\bm q-\bm q_h\|_{\mathcal T_h}&\lesssim h^{k}(|\bm q|_{k+1}+|y|_{k+1}+|\bm p|_{k+1}+|z|_{k+1}),\\
	\|\bm p-\bm p_h\|_{\mathcal T_h}&\lesssim h^{k}(|\bm q|_{k+1}+|y|_{k+1}+|\bm p|_{k+1}+|z|_{k+1}),\\
	\|y-y_h\|_{\mathcal T_h}&\lesssim h^{k+1}(|\bm q|_{k+1}+|y|_{k+1}+|\bm p|_{k+1}+|z|_{k+1}),\\
	\|z-z_h\|_{\mathcal T_h}&\lesssim h^{k+1}(|\bm q|_{k+1}+|y|_{k+1}+|\bm p|_{k+1}+|z|_{k+1}),\\
	\|u-u_h\|_{\mathcal T_h}&\lesssim h^{k+1}(|\bm q|_{k+1}+|y|_{k+1}+|\bm p|_{k+1}+|z|_{k+1}).
	\end{align*}
\end{theorem}

\subsection{Preliminary material}
\label{sec:Projectionoperator}

The convergence analysis of the EDG method for the Poisson problem without control has been performed in \cite{MR2551142}.  The authors of \cite{MR2551142} use a special projection to split the errors are prove the convergence.  We do not use the special projection from \cite{MR2551142} in our analysis; instead, we use the standard $L^2$-orthogonal projection operators $\bm{\Pi}_V$ and $\Pi_W$ satisfying
\begin{subequations} \label{def_L2}
	\begin{align}
			(\bm \Pi_V \bm q, \bm r)_{K}&=(\bm q,\bm r)_{K} \quad \forall \bm r\in \bm{\mathcal{ P}}_k(K),\\
	(\Pi_W y,w)_{K}&=(y,w)_{K}\quad \forall w\in \mathcal{P}_{k}(K).
	\end{align}
\end{subequations}
In the conclusion, we briefly mention future work connected to the different EDG analysis approach taken here.

We use the following well-known bounds:
\begin{subequations}\label{classical_ine}
	\begin{align}
	\norm {\bm q -\bm\Pi_V \bm q}_{\mathcal T_h} &\le  Ch^{k+1} \norm{\bm q}_{k+1,\Omega}, \ \, \norm {y -{\Pi_W y}}_{\mathcal T_h} \le C h^{k+1} \norm{y}_{k+1,\Omega},\\
	\norm {y -{\Pi_W y}}_{\partial\mathcal T_h} &\le C  h^{k+\frac 1 2} \norm{y}_{k+1,\Omega},
	 \  \norm {\bm q -\bm\Pi_V \bm q}_{\partial \mathcal T_h} \le C h^{k+\frac 12} \norm{\bm q}_{k+1,\Omega},\\
		\norm {y -{ I_h y}}_{\partial\mathcal T_h} &\le C  h^{k+\frac 1 2} \norm{y}_{k+1,\Omega}, \  \norm {w}_{\partial \mathcal T_h} \le C h^{-\frac 12} \norm {w}_{ \mathcal T_h}, \forall w\in W_h.
	\end{align}
\end{subequations}
where  $I_h $ is a continuous interpolation operator, and we have the same projection error bounds for $\bm p$ and $z$.

Next, define the EDG operator $ \mathscr B$ by
\begin{equation}\label{def_B1}
\begin{split}
\hspace{1em}&\hspace{-1em}  \mathscr  B( \bm v_h,w_h,\mu_h;\bm r_1,w_1,\mu_1) \\
&=(\bm q_h,\bm r_1)_{\mathcal T_h}-( y_h,\nabla\cdot\bm r_1)_{\mathcal T_h}+\langle \widehat y_h^o,\bm r_1\cdot\bm n\rangle_{\partial\mathcal T_h\backslash \varepsilon_h^\partial}-(\bm q_h,  \nabla w_1)_{\mathcal T_h} \\
&\quad +\langle {\bm q}_h\cdot\bm n +h^{-1} y_h,w_1\rangle_{\partial\mathcal T_h}-\langle h^{-1}\widehat{y}_h^o,w_1 \rangle_{\partial \mathcal{T}_h\backslash \varepsilon_h^\partial}\\
&\quad-\langle  {\bm q}_h\cdot\bm n+h^{-1}(y_h-\widehat y_h^o),\mu_1\rangle_{\partial\mathcal T_h\backslash\varepsilon^{\partial}_h}.
\end{split}
\end{equation}
By the definition in \eqref{def_B1},  we can rewrite the EDG formulation of the optimality system \eqref{HDG_discrete2} as follows: find $({\bm{q}}_h,{\bm{p}}_h,y_h,z_h,u_h,\widehat y_h^o,\widehat z_h^o)\in \bm{V}_h\times\bm{V}_h\times W_h \times W_h\times W_h\times \widetilde{M}_h(o)\times \widetilde{M}_h(o)$  such that
\begin{subequations}\label{EDG_full_discrete}
	\begin{align}
	\mathscr B(\bm q_h,y_h,\widehat y_h^o;\bm r_1,w_1,\mu_1)&=( f+ u_h, w_1)_{\mathcal T_h} + \langle I_hg, h^{-1} w_1-\bm r_1\cdot\bm n \rangle_{\varepsilon_h^\partial},\label{EDG_full_discrete_a}\\
	\mathscr B(\bm p_h,z_h,\widehat z_h^o;\bm r_2,w_2,\mu_2)&=(y_d-y_h,w_2)_{\mathcal T_h},\label{EDG_full_discrete_b}\\
	(z_h-\gamma u_h,w_3)_{\mathcal T_h}&= 0,\label{EDG_full_discrete_e}
	\end{align}
\end{subequations}
for all $\left(\bm{r}_1, \bm{r}_2,w_1,w_2,w_3,\mu_1,\mu_2\right)\in \bm{V}_h\times\bm{V}_h\times W_h \times W_h\times W_h\times \widetilde{M}_h(o)\times \widetilde{M}_h(o)$.

Below, we present two fundamental properties of the operator $\mathscr B$, and show the EDG discretization of the optimality system \eqref{EDG_full_discrete} has a unique solution.  The strategy of the proofs of these three results is similar to our earlier HDG work \cite{HuShenSinglerZhangZheng_HDG_Dirichlet_control2}; we include the proofs to make this paper self-contained.
\begin{lemma}\label{property_B}
	For any $ ( \bm v_h, w_h, \mu_h ) \in \bm V_h \times W_h \times M_h(o) $, we have
	\begin{align*}
	\hspace{1em}&\hspace{-1em} \mathscr B(\bm v_h,w_h,\mu_h;\bm v_h,w_h,\mu_h)\\
	&=(\bm v_h,\bm v_h)_{\mathcal T_h}+ \langle h^{-1} (w_h-\mu_h),w_h-\mu_h\rangle_{\partial\mathcal T_h\backslash \varepsilon_h^\partial}+\langle h^{-1} w_h,w_h\rangle_{\varepsilon_h^\partial}.
	%%%%%%%%%%%%%%%%%%%
	\end{align*}
\end{lemma}
\begin{proof}  Compute:
	\begin{align*}
	%%%%%%%%%%%%%%%
	\hspace{1em}&\hspace{-1em} 	\mathscr B(\bm v_h,w_h,\mu_h;\bm v_h,w_h,\mu_h)\\
	&=(\bm v_h,\bm v_h)_{\mathcal T_h}-( w_h,\nabla\cdot\bm v_h)_{\mathcal T_h}+\langle \mu_h,\bm v_h\cdot\bm n\rangle_{\partial\mathcal T_h\backslash \varepsilon_h^\partial}-(\bm v_h,  \nabla w_h)_{\mathcal T_h}\\
	& \quad  +\langle {\bm v}_h\cdot\bm n +h^{-1}  w_h,w_h\rangle_{\partial\mathcal T_h}-\langle h^{-1} \mu_h,w_h\rangle_{\partial\mathcal T_h\backslash \varepsilon_h^\partial} \\
	& \quad-\langle  {\bm v}_h\cdot\bm n+ h^{-1} (w_h - \mu_h),\mu_h \rangle_{\partial\mathcal T_h\backslash\varepsilon^{\partial}_h}\\
	&=(\bm v_h,\bm v_h)_{\mathcal T_h}+\langle  h^{-1}  w_h,w_h\rangle_{\partial\mathcal T_h}-\langle h^{-1}  \mu_h,w_h\rangle_{\partial\mathcal T_h\backslash \varepsilon_h^\partial}\\
	&\quad -\langle h^{-1} (w_h-\mu_h),\mu_h \rangle_{\partial\mathcal T_h\backslash \varepsilon_h^\partial}\\
	&=(\bm v_h,\bm v_h)_{\mathcal T_h}+ \langle h^{-1}  (w_h-\mu_h),w_h-\mu_h\rangle_{\partial\mathcal T_h\backslash \varepsilon_h^\partial}+\langle h^{-1}  w_h,w_h\rangle_{\varepsilon_h^\partial}.
	\end{align*}
\end{proof}

% Next, we give a property of the EDG operator $\mathscr B$ that is critical to our error analysis of the method.
\begin{lemma}\label{identical_equa}
	We have
	$$\mathscr B (\bm q_h,y_h,\widehat y_h^o;\bm p_h,-z_h,-\widehat z_h^o) + \mathscr B (\bm p_h,z_h,\widehat z_h^o;-\bm q_h,y_h,\widehat y_h^o) = 0.$$
\end{lemma}
\begin{proof}
	By the definition of $ \mathscr B $, and integration by parts:
	\begin{align*}
	\hspace{1em}&\hspace{-1em}  \mathscr B (\bm q_h,y_h,\widehat y_h^o;\bm p_h,-z_h,-\widehat z_h^o) + \mathscr B (\bm p_h,z_h,\widehat z_h^o;-\bm q_h,y_h,\widehat y_h^o)\\
	&=(\bm{q}_h, \bm p_h)_{{\mathcal{T}_h}}- (y_h, \nabla\cdot \bm p_h)_{{\mathcal{T}_h}}+\langle \widehat{y}_h^o, \bm p_h\cdot \bm{n} \rangle_{\partial{{\mathcal{T}_h}}\backslash {\varepsilon_h^{\partial}}} \\
	&\quad + (\bm{q}_h , \nabla z_h)_{{\mathcal{T}_h}}  - \langle\bm q_h\cdot\bm n +h^{-1}  y_h , z_h \rangle_{\partial{{\mathcal{T}_h}}}  + \langle h^{-1}  \widehat y_h^o, z_h \rangle_{\partial{{\mathcal{T}_h}}\backslash \varepsilon_h^{\partial}} \\
	&\quad+ \langle\bm q_h\cdot\bm n +h^{-1} (y_h-\widehat y_h^o), \widehat z_h^o  \rangle_{\partial{{\mathcal{T}_h}}\backslash\varepsilon_h^{\partial}}\\
	&\quad-(\bm{p}_h, \bm q_h)_{{\mathcal{T}_h}}+ (z_h, \nabla\cdot \bm q_h)_{{\mathcal{T}_h}} -\langle \widehat{z}_h^o, \bm q_h \cdot \bm{n} \rangle_{\partial{{\mathcal{T}_h}}\backslash {\varepsilon_h^{\partial}}}\\
	&\quad - (\bm{p}_h , \nabla y_h)_{{\mathcal{T}_h}} +\langle\bm p_h\cdot\bm n +h^{-1} z_h , y_h \rangle_{\partial{{\mathcal{T}_h}}} - \langle h^{-1} \widehat z_h^o, y_h \rangle_{\partial{{\mathcal{T}_h}}\backslash \varepsilon_h^{\partial}}\\
	&\quad- \langle\bm p_h\cdot\bm n + h^{-1}  (z_h-\widehat z_h^o), \widehat y_h^o \rangle_{\partial{{\mathcal{T}_h}}\backslash\varepsilon_h^{\partial}}\\
	&=0.
	\end{align*}
\end{proof}

\begin{proposition}\label{ex_uni}
	There exists a unique solution of the HDG equations \eqref{EDG_full_discrete}.
\end{proposition}
\begin{proof}
	Since the system \eqref{EDG_full_discrete} is finite dimensional, we only need to prove solutions are unique.  To do this, we show zero is the only solution of the system \eqref{EDG_full_discrete} for problem data $y_d = f =g= 0$.
	
	Take $(\bm r_1,w_1,\mu_1) = (\bm p_h,-z_h,-\widehat z_h^o)$, $(\bm r_2,w_2,\mu_2) = (-\bm q_h,y_h,\widehat y_h^o)$, and $w_3 = z_h -\gamma u_h $ in the EDG equations \eqref{EDG_full_discrete_a},  \eqref{EDG_full_discrete_b}, and \eqref{EDG_full_discrete_e}, respectively, and sum to obtain
	\begin{align*}
	\hspace{3em}&\hspace{-3em} \mathscr B  (\bm q_h,y_h,\widehat y_h^o;\bm p_h,-z_h,-\widehat z_h^o) + \mathscr B (\bm p_h,z_h,\widehat z_h^o;-\bm q_h,y_h,\widehat y_h^o) \\
	& =	- (y_h,y_h)_{\mathcal T_h} -  \gamma^{-1} (z_h,z_h)_{\mathcal T_h}.% \\
	%&  = 0.
	\end{align*}
	Since $\gamma>0$,  \Cref{identical_equa} implies $y_h =  u_h = z_h= 0$.
	
	Next, take $(\bm r_1,w_1,\mu_1) = (\bm q_h,y_h,\widehat y_h^o)$ and $(\bm r_2,w_2,\mu_2) = (\bm p_h,z_h,\widehat z_h^o)$ in the EDG equations \eqref{EDG_full_discrete_a}-\eqref{EDG_full_discrete_b}.   \Cref{property_B} gives $\bm q_h= \bm p_h= \bm 0 $ and $ \widehat y_h^o  = \widehat z_h^o=0$.
\end{proof}

\subsection{Proof of Main Result}
For our proof of the convergence results, we follow the strategy in \cite{HuShenSinglerZhangZheng_HDG_Dirichlet_control1} and consider the EDG discretization of the optimality system with the exact optimal control fixed.  This results in the following auxiliary problem: find $$({\bm{q}}_h(u),{\bm{p}}_h(u), y_h(u), z_h(u), {\widehat{y}}_h^o(u), {\widehat{z}}_h^o(u))\in \bm{V}_h\times\bm{V}_h\times W_h \times W_h\times \widetilde{M}_h(o)\times \widetilde{M}_h(o)$$ satisfying
\begin{subequations}\label{HDG_inter_u}
	\begin{align}
	\mathscr B(\bm q_h(u),y_h(u),\widehat{y}_h(u);\bm r_1, w_1,\mu_1)&=( f+ u,w_1)_{\mathcal T_h} \ \nonumber\\
	& \quad+ \langle I_hg, h^{-1} w_1-\bm r_1\cdot\bm n \rangle_{\varepsilon_h^\partial},\label{EDG_u_a} \\
	\mathscr B(\bm p_h(u),z_h(u),\widehat{z}_h(u);\bm r_2, w_2,\mu_2)&=(y_d-y_h(u), w_2)_{\mathcal T_h},\label{EDG_u_b}
	\end{align}
\end{subequations}
for all $\left(\bm{r}_1, \bm{r}_2,w_1,w_2,\mu_1,\mu_2\right)\in \bm{V}_h\times\bm{V}_h \times W_h\times W_h\times \widetilde{M}_h(o)\times \widetilde{M}_h(o)$.

We split our proof into seven steps, and estimate the errors between the solutions of the exact optimality system, the auxiliary problem, and the EDG discretization of the optimality system.

We start with the auxiliary problem and the mixed formulation of the optimality system \eqref{mixed_a}-\eqref{mixed_d}.  In Steps 1-3 below, we estimate the errors in the state $ y $ and the flux $ \bm{q} $.  We split the errors with the $ L^2 $ projections and the continuous interpolation operator.  We use the following notation:
\begin{align}\label{notation}
\begin{aligned}%column 1
\delta^{\bm q} &=\bm q-{\bm\Pi}_V\bm q,\\
\delta^y&=y- \Pi_W y,\\
\delta^{\widehat y} &= y-I_h y,\\
\widehat {\bm\delta}_1 &= \delta^{\bm q}\cdot\bm n+  h^{-1} (\delta^y- \delta^{\widehat{y}}),
\end{aligned}
&&
\begin{aligned}%column 2
\varepsilon^{\bm q}_h &= {\bm\Pi}_V \bm q-\bm q_h(u),\\
\varepsilon^{y}_h &= \Pi_W y-y_h(u),\\
\varepsilon^{\widehat y}_h &= I_h y-\widehat{y}_h(u),\\
\widehat {\bm \varepsilon }_1 &= \varepsilon_h^{\bm q}\cdot\bm n+h^{-1} (\varepsilon^y_h-\varepsilon_h^{\widehat y}).
\end{aligned}
\end{align}
where $\widehat y_h(u) = \widehat y_h^o(u)$ on $\varepsilon_h^o$ and $\widehat y_h(u) = I_h g$ on $\varepsilon_h^{\partial}$, which implies $\varepsilon_h^{\widehat y} = 0$ on $\varepsilon_h^{\partial}$.

\subsubsection{Step 1: The error equation for part 1 of the auxiliary problem \eqref{EDG_u_a}.} \label{subsec:proof_step1}
\begin{lemma} \label{lemma:step_1}
	We have
	\begin{align} \label{error_y}
	\hspace{3em}\hspace{-3em} \mathscr B(\varepsilon^{\bm q}_h,\varepsilon^y_h,\varepsilon^{\widehat{y}}_h;\bm r_1, w_1,\mu_1) =-\langle \delta^{\widehat{y}},\bm r_1\cdot\bm n \rangle_{\partial \mathcal{T}_h}-\langle\widehat{\bm \delta}_1,w_1\rangle_{\partial \mathcal{T}_h}+\langle \widehat{\bm \delta}_1,\mu_1 \rangle_{\partial \mathcal{T}_h\backslash\varepsilon_h^\partial}.
	\end{align}
\end{lemma}
\begin{proof}
	By the definition of the EDG operator $\mathscr B$ in \eqref{def_B1}, we have
	\begin{align*}
	\hspace{1em}&\hspace{-1em} \mathscr B ({\bm \Pi}_V\bm q,{\Pi}_W y,I_h y;\bm r_1,w_1,\mu_1) \\
	& =({\bm \Pi}_V\bm q,\bm r_1)_{\mathcal T_h}-({\Pi}_W y,\nabla\cdot\bm r_1)_{\mathcal T_h}+\langle I_h y,\bm r_1\cdot\bm n\rangle_{\partial\mathcal T_h\backslash\varepsilon_h^\partial}\\
	&\quad-({\bm \Pi}_V\bm q ,  \nabla w_1)_{\mathcal T_h} +\langle {\bm \Pi}_V\bm q \cdot\bm n +h^{-1}  {\Pi}_W y ,w_1\rangle_{\partial\mathcal T_h}-\langle h^{-1}  I_h y ,w_1\rangle_{\partial\mathcal T_h\backslash\varepsilon_h^\partial}\nonumber\\
	&\quad-\langle  {\bm \Pi}_V\bm q \cdot\bm n+h^{-1} ( {\Pi}_W y -I_h y ),\mu_1\rangle_{\partial\mathcal T_h\backslash\varepsilon^{\partial}_h}.
	\end{align*}
	Using the properties of the $L^2$-orthogonal projections \eqref{def_L2} gives
	\begin{align*}
		\hspace{1em}&\hspace{-1em} \mathscr B ({\bm \Pi}_V\bm q,{\Pi}_W y,I_h y;\bm r_1,w_1,\mu_1) \\
		&=(\bm q,\bm r_1)_{\mathcal T_h}-( y,\nabla\cdot\bm r_1)_{\mathcal T_h}+\langle  y,\bm r_1\cdot\bm n\rangle_{\partial\mathcal T_h\backslash\varepsilon_h^\partial}-\langle \delta^{\widehat{y}},\bm r_1\cdot\bm n \rangle_{\partial\mathcal T_h\backslash\varepsilon_h^\partial}\\
		&\quad-(\bm q,\nabla w_1)_{\mathcal{T}_h}+\langle \bm q\cdot\bm n,w_1 \rangle_{\partial \mathcal{T}_h}+\langle h^{-1}  y,w_1 \rangle_{\varepsilon_h^\partial}-\langle \delta^{\bm q}\cdot\bm n+h^{-1} \delta^y,w_1 \rangle_{\partial \mathcal{T}_h}\\
		&\quad+\langle h^{-1}  \delta^{\widehat{y}},w_1 \rangle_{\partial\mathcal T_h\backslash\varepsilon_h^\partial}-\langle \bm q\cdot \bm n,\mu_1 \rangle_{\partial\mathcal T_h\backslash\varepsilon_h^\partial} +\langle \widehat{\bm \delta}_1,\mu_1 \rangle_{\partial\mathcal T_h\backslash\varepsilon_h^\partial}.%\\
		%
		%
		%
		%&\quad -\langle \bm q\cdot \bm n,\mu_1 \rangle_{\partial\mathcal T_h\backslash\varepsilon_h^\partial} +\langle \widehat{\bm \delta}_1,\mu_1 \rangle_{\partial\mathcal T_h\backslash\varepsilon_h^\partial}
	\end{align*}
		The exact solution $\bm q$ and $y$ satisfies
	\begin{align*}
	(\bm q,\bm r_1)_{\mathcal T_h}-(y,\nabla\cdot\bm r_1)_{\mathcal T_h}+\langle  y,\bm r_1\cdot\bm n\rangle_{\partial\mathcal T_h}&=0,\\
	-(\bm q , \nabla w_1)_{\mathcal T_h}+\langle {\bm q}\cdot\bm n,w_1\rangle_{\partial\mathcal T_h}&= (f+u, w_1)_{\mathcal T_h},\\
	\langle {\bm q}\cdot\bm n,\mu_1\rangle_{\partial\mathcal T_h\backslash\varepsilon^{\partial}_h}&=0,
	\end{align*}
	and therefore
	\begin{align*}
	\hspace{1em}&\hspace{-1em} \mathscr B ({\bm \Pi}_V\bm q,{\Pi}_W y,I_h y;\bm r_1,w_1,\mu_1) \\
		&=-\langle y,\bm r_1\cdot\bm n \rangle_{\varepsilon_h^\partial}-\langle \delta^{\widehat{y}},\bm r_1\cdot\bm n \rangle_{\partial\mathcal T_h\backslash\varepsilon_h^\partial}+(f+u_h,w_1)_{\mathcal{T}_h}+\langle h^{-1}  y,w_1 \rangle_{\varepsilon_h^\partial}\\
		&\quad-\langle \delta^{\bm q}\cdot\bm n+h^{-1}  \delta^y,w_1 \rangle_{\partial \mathcal{T}_h}+\langle h^{-1}  \delta^{\widehat{y}},w_1 \rangle_{\partial\mathcal T_h\backslash\varepsilon_h^\partial}+\langle \widehat{\bm \delta}_1,\mu_1 \rangle_{\partial\mathcal T_h\backslash\varepsilon_h^\partial}
	\end{align*}
	Subtracting equation \eqref{EDG_u_a} from the above equation completes the proof.
\end{proof}

\subsubsection{Step 2: Estimate for $\varepsilon_h^{ q}$.}
\label{subsec:proof_step2}

\begin{lemma} \label{energy_norm_q}
	We have
	\begin{align}
		\|\varepsilon_h^{\bm q}\|_{\mathcal{T}_h}+h^{-\frac{1}{2}}\|\varepsilon_h^y-\varepsilon_h^{\widehat{y}}\|_{\partial \mathcal{T}_h} \lesssim h^{k} (|\bm q|_{k+1}+|y|_{k+1}).
	\end{align}
\end{lemma}
\begin{proof}
	Take $(\bm r_1,w_1,\mu_1)=(\bm \varepsilon_h^{\bm q},\varepsilon_h^y,\varepsilon_h^{\widehat{y}})$ in equation \eqref{error_y} and use $\varepsilon_h^{\widehat{y}}=0$ on $\varepsilon_h^\partial$ to get
	\begin{align*}
	\hspace{3em}\hspace{-3em} \mathscr B(\varepsilon^{\bm q}_h,\varepsilon^y_h,\varepsilon^{\widehat{y}}_h;\bm \varepsilon_h^{\bm q},\varepsilon_h^y,\varepsilon_h^{\widehat{y}}) &=-\langle \delta^{\widehat{y}},\varepsilon_h^{\bm q}\cdot\bm n \rangle_{\partial \mathcal{T}_h}-\langle\widehat{\bm \delta}_1,\varepsilon_h^y\rangle_{\partial \mathcal{T}_h}+\langle \widehat{\bm \delta}_1,\varepsilon_h^{\widehat{y}} \rangle_{\partial \mathcal{T}_h\backslash\varepsilon_h^\partial}\\
	&=-\langle \delta^{\widehat{y}},\varepsilon_h^{\bm q}\cdot\bm n \rangle_{\partial \mathcal{T}_h}-\langle\widehat{\bm \delta}_1,\varepsilon_h^y-\varepsilon_h^{\widehat{y}}\rangle_{\partial \mathcal{T}_h}.
	\end{align*}
	%Here, we used $\varepsilon_h^{\widehat{y}}=0$ on $\varepsilon_h^\partial$. Next, we have
	Next, we have
	\begin{align*}
		-\langle \delta^{\widehat{y}},\varepsilon_h^{\bm q}\cdot\bm n \rangle_{\partial \mathcal{T}_h}&\le C\| \delta^{\widehat{y}} \|_{\partial \mathcal{T}_h}\|\varepsilon_h^{\bm q}\|_{\partial \mathcal{T}_h} \le C h^{-\frac{1}{2}} \| \delta^{\widehat{y}} \|_{\partial \mathcal{T}_h}\| \varepsilon_h^{\bm q}\|_{\mathcal{T}_h},\\
		\langle\widehat{\bm \delta}_1,\varepsilon_h^y-\varepsilon_h^{\widehat{y}}\rangle_{\partial \mathcal{T}_h}&=-\langle \delta^{\bm q}\cdot\bm n+\frac{1}{h}(\delta^y-\delta^{\widehat{y}}),\varepsilon_h^y-\varepsilon_h^{\widehat{y}}\rangle_{\partial \mathcal{T}_h}\\
		&\le (\|\delta^{\bm q}\|_{\partial \mathcal{T}_h}+h^{-1}\| \delta^y \|_{\partial \mathcal{T}_h}+h^{-1}\| \delta^{\widehat{y}} \|_{\partial \mathcal{T}_h})\| \varepsilon_h^y-\varepsilon_h^{\widehat{y}} \|_{\partial \mathcal{T}_h}\\
		&=(h^\frac{1}{2}\|\delta^{\bm q}\|_{\partial \mathcal{T}_h}+h^{-\frac{1}{2}}\| \delta^y \|_{\partial \mathcal{T}_h}+h^{-\frac{1}{2}}\| \delta^{\widehat{y}} \|_{\partial \mathcal{T}_h})h^{-\frac{1}{2}}\| \varepsilon_h^y-\varepsilon_h^{\widehat{y}} \|_{\partial \mathcal{T}_h}.
	\end{align*}
	The energy property of operator $\mathscr B$ in \Cref{property_B} gives
	\begin{align*}
		\|\varepsilon_h^{\bm q}\|_{\mathcal{T}_h}+h^{-\frac{1}{2}}\|\varepsilon_h^y-\varepsilon_h^{\widehat{y}}\|_{\partial \mathcal{T}_h} &\lesssim  h^{-\frac{1}{2}}\| \delta^{\widehat{y}} \|_{\partial \mathcal{T}_h}+h^{-\frac{1}{2}}\|\delta^y\|_{\partial \mathcal{T}_h}+h^\frac{1}{2}\|\delta^{\bm q}\|_{\partial \mathcal{T}_h}\\
		&\lesssim h^{k} (|\bm q|_{k+1}+|y|_{k+1}).
	\end{align*}
\end{proof}

\subsubsection{Step 3: Estimate for $\varepsilon_h^y$ by a duality argument.} \label{subsec:proof_step_3}
Next, we introduce the dual problem for any given $\Theta$ in $L^2(\Omega)$:
\begin{equation}\label{Dual_PDE}
\begin{split}
\bm\Phi+\nabla\Psi&=0\qquad~\text{in}\ \Omega,\\
\nabla\cdot\bm{\Phi}&=\Theta \qquad\text{in}\ \Omega,\\
\Psi&=0\qquad~\text{on}\ \partial\Omega.
\end{split}
\end{equation}
Since the domain $\Omega$ is convex, we have the regularity estimate
\begin{align}\label{dual_esti}
\norm{\bm \Phi}_{1,\Omega} + \norm{\Psi}_{2,\Omega} \le C_{\text{reg}} \norm{\Theta}_\Omega.
\end{align}
In the proof below for estimating $\varepsilon_h^{y}$, we use the following notation:
\begin{align}
\delta^{\bm \Phi} &=\bm \Phi-{\bm\Pi}_V\bm \Phi, \quad \delta^\Psi=\Psi- {\Pi}_W \Psi, \quad
\delta^{\widehat \Psi} = \Psi-I_h \Psi.\label{notation_2}
\end{align}
\begin{lemma} \label{dual_y}
	We have
	\begin{align}
		\| \varepsilon_h^y \|_{\mathcal{T}_h} \lesssim h^{k+1} (|\bm q|_{k+1}+|y|_{k+1}).
	\end{align}
\end{lemma}
\begin{proof}
	First, take $(\bm r_1,w_1,\mu_1)=(\bm \Pi_V \bm \Phi, -\Pi_W \Psi,-I_h \Psi)$ in equation \eqref{error_y} to get
	\begin{align*}
\hspace{1em}&\hspace{-1em} \mathscr B(\varepsilon^{\bm q}_h,\varepsilon^y_h,\varepsilon^{\widehat{y}}_h;\bm \Pi_V \bm \Phi, -\Pi_W \Psi,-I_h \Psi)\\
&=(\varepsilon_h^{\bm q},\bm \Pi_V \bm \Phi)_{\mathcal T_h}-( \varepsilon_h^y,\nabla\cdot\bm \Pi_V \bm \Phi)_{\mathcal T_h}+\langle \varepsilon_h^{\widehat{y}},\bm \Pi_V \bm \Phi\cdot\bm n\rangle_{\partial\mathcal T_h\backslash \varepsilon_h^\partial} \\
&\quad +(\varepsilon_h^{\bm q},  \nabla \Pi_W \Psi)_{\mathcal T_h}-\langle \varepsilon_h^{\bm q}\cdot\bm n +h^{-1}  \varepsilon_h^y,\Pi_W \Psi\rangle_{\partial\mathcal T_h}+\langle h^{-1} \varepsilon_h^{\widehat{y}},\Pi_W \Psi \rangle_{\partial \mathcal{T}_h\backslash \varepsilon_h^\partial}\\
&\quad+\langle  \varepsilon_h^{\bm q}\cdot\bm n+h^{-1}(\varepsilon_h^y-\varepsilon_h^{\widehat{y}}),I_h \Psi\rangle_{\partial\mathcal T_h\backslash\varepsilon^{\partial}_h}.
	\end{align*}
	Next, integration by parts gives
	\begin{align*}
	-(\varepsilon_h^y,\nabla\cdot \bm \Pi_V \bm \Phi)_{\partial \mathcal{T}_h}&=(\nabla \varepsilon_h^y,\bm \Phi)_{\mathcal{T}_h}-\langle \varepsilon_h^y,\bm \Pi_V\Phi \cdot \bm n\rangle_{\partial \mathcal{T}_h}\\
	&=-(\varepsilon_h^y,\nabla\cdot \bm \Phi)_{\mathcal{T}_h}+\langle \varepsilon_h^y,\delta^{\bm \Phi} \cdot \bm n\rangle_{\partial \mathcal{T}_h},\\
	( \varepsilon_h^q,\nabla \Pi_W \Psi)_{\mathcal{T}_h}&=-(\nabla\cdot \varepsilon_h^q, \Psi)_{\mathcal{T}_h}+\langle \varepsilon_h^{\bm q}\cdot\bm n,\Pi_W \Psi \rangle_{\partial \mathcal{T}_h}\\
	&=(\varepsilon_h^{\bm q},\nabla \Psi)_{\mathcal{T}_h}-\langle \varepsilon_h^{\bm q}\cdot\bm n,\delta^\Psi \rangle_{\partial \mathcal{T}_h}.
	\end{align*}
	Since $ \bm \Phi $ and $ \Psi $ satisfy the dual problem \eqref{Dual_PDE} with $\Theta=-\varepsilon_h^y$, we obtain
	\begin{align*}
		\hspace{1em}&\hspace{-1em} \mathscr B(\varepsilon^{\bm q}_h,\varepsilon^y_h,\varepsilon^{\widehat{y}}_h;\bm \Pi_V \bm \Phi, -\Pi_W \Psi,-I_h \Psi)\\
		&=(\varepsilon_h^{\bm q}, \bm \Phi)_{\mathcal T_h}-( \varepsilon_h^y,\nabla\cdot \bm \Phi)_{\mathcal T_h}+\langle \varepsilon_h^y-\varepsilon_h^{\widehat{y}},\delta^{\bm \Phi}\cdot\bm n\rangle_{\partial\mathcal T_h} \\
		&\quad +(\varepsilon_h^{\bm q},  \nabla \Psi)_{\mathcal T_h}-\langle \varepsilon_h^{\bm q}\cdot\bm n,\Psi \rangle_{\partial \mathcal{T}_h}-\langle h^{-1} \varepsilon_h^y,\Pi_W \Psi\rangle_{\partial\mathcal T_h}\\
		&\quad+\langle h^{-1}  \varepsilon_h^{\widehat{y}},\Pi_W \Psi \rangle_{\partial \mathcal{T}_h}+\langle  \varepsilon_h^{\bm q}\cdot\bm n+h^{-1} (\varepsilon_h^y-\varepsilon_h^{\widehat{y}}),I_h \Psi\rangle_{\partial\mathcal T_h}\\
		&=(\varepsilon_h^y,\varepsilon_h^y)_{\mathcal{T}_h}+\langle \varepsilon_h^y-\varepsilon_h^{\widehat{y}},\delta^{\bm \Phi}\cdot \bm n \rangle_{\partial \mathcal{T}_h}-\langle \varepsilon_h^{\bm q}\cdot\bm n,\delta^{\widehat{\Psi}} \rangle_{\partial \mathcal{T}_h}\\
		&\quad +\langle h^{-1}  (\varepsilon_h^y-\varepsilon_h^{\widehat{y}}),\delta^\Psi-\delta^{\widehat{\Psi}} \rangle_{\partial \mathcal{T}_h},
	\end{align*}
	where we used $\langle \varepsilon_h^{\widehat{y}},\bm \Phi\cdot \bm n \rangle_{\partial \mathcal{T}_h\backslash\varepsilon_h^\partial}=0$ and $\Psi= \delta^{\widehat{y}}=0$ on $\varepsilon_h^\partial$.
	
	On the other hand, from equation \eqref{error_y} and $\langle \delta^{\widehat{y}},{\bm \Phi}\cdot\bm n \rangle_{\partial \mathcal{T}_h} =0$ we have
	\begin{align*}
	\hspace{4em}&\hspace{-4em}  \mathscr B(\varepsilon^{\bm q}_h,\varepsilon^y_h,\varepsilon^{\widehat{y}}_h;\bm \Pi_V \bm \Phi, -\Pi_W \Psi,-I_h \Psi) \\
	&=-\langle \delta^{\widehat{y}},\bm \Pi_V \bm \Phi\cdot\bm n \rangle_{\partial \mathcal{T}_h}+\langle\widehat{\bm \delta}_1,\delta^\Psi-\delta^{\widehat{\Psi}}\rangle_{\partial \mathcal{T}_h}\\
	&=\langle \delta^{\widehat{y}},\delta^{\bm \Phi}\cdot\bm n \rangle_{\partial \mathcal{T}_h}+\langle\widehat{\bm \delta}_1,\delta^\Psi-\delta^{\widehat{\Psi}}\rangle_{\partial \mathcal{T}_h}.
	\end{align*}
	Comparing with the two equations above, we have
	\begin{align*}
		\|\varepsilon_h^y\|_{\mathcal{T}_h}^2&=\langle \delta^{\widehat{y}},\delta^{\bm \Phi}\cdot\bm n \rangle_{\partial \mathcal{T}_h}+\langle\widehat{\bm \delta}_1,\delta^\Psi-\delta^{\widehat{\Psi}}\rangle_{\partial \mathcal{T}_h}-\langle \varepsilon_h^y-\varepsilon_h^{\widehat{y}},\delta^{\bm \Phi}\cdot \bm n \rangle_{\partial \mathcal{T}_h}\\
		&\quad +\langle \varepsilon_h^{\bm q}\cdot\bm n,\delta^{\widehat{\Psi}} \rangle_{\partial \mathcal{T}_h}-\langle h^{-1} (\varepsilon_h^y-\varepsilon_h^{\widehat{y}}),\delta^\Psi-\delta^{\widehat{\Psi}} \rangle_{\partial \mathcal{T}_h}\\
		&=:T_1+T_2+T_3+T_4+T_5.
	\end{align*}
	By the Cauchy-Schwarz inequality, \Cref{energy_norm_q}, and \eqref{classical_ine}, we have
	\begin{align*}
		T_1 &\le \|\delta^{\widehat{y}}\|_{\partial \mathcal{T}_h} \|\delta^{\bm \Phi}\|_{\partial \mathcal{T}_h}\lesssim h^\frac{1}{2} \|\delta^{\widehat{y}}\|_{\partial \mathcal{T}_h}\|\bm \Phi\|_{1,\Omega}\lesssim h^\frac{1}{2} \|\delta^{\widehat{y}}\|_{\partial \mathcal{T}_h}\|\varepsilon_h^y\|_{\Omega},\\
		&\lesssim h^{k+1}(|\bm q|_{k+1}+|y|_{k+1})\| \varepsilon_h^y \|_{\mathcal{T}_h},\\
		T_2 &\lesssim h^{\frac{3}{2}} (\|\delta^{\bm q}\|_{\partial \mathcal{T}_h}+h^{-1} \|\delta^y-\delta^{\widehat{y}}\|_{\partial \mathcal{T}_h})\|\Psi\|_{2,\Omega}\\
		&\lesssim  ( h\|\delta^{\bm q}\|_{\mathcal{T}_h} +h^{\frac{1}{2}}(\|\delta^y\|_{\partial \mathcal{T}_h}+\| \delta^{\widehat{y}} \|_{\partial \mathcal{T}_h}))\| \varepsilon_h^y \|_{\mathcal{T}_h}\\
		&\lesssim h^{k+1}(|\bm q|_{k+1}+|y|_{k+1})\| \varepsilon_h^y \|_{\mathcal{T}_h},\\
		T_3&\le \|\varepsilon_h^y-\varepsilon_h^{\widehat{y}}\|_{\partial \mathcal{T}_h}\| \delta^{\bm \Phi} \|_{\partial \mathcal{T}_h}\lesssim h^\frac{1}{2} \|\varepsilon_h^y-\varepsilon_h^{\widehat{y}}\|_{\partial \mathcal{T}_h} \|\bm \Phi\|_{1,\Omega}\\
		&\lesssim h^{1/2} \|\varepsilon_h^y-\varepsilon_h^{\widehat{y}}\|_{\partial \mathcal{T}_h} \|\varepsilon_h^y\|_{\mathcal{T}_h}\lesssim h^{k+1}(|\bm q|_{k+1}+|y|_{k+1})\| \varepsilon_h^y \|_{\mathcal{T}_h},\\
		T_4 &\lesssim h\|\varepsilon_h^{\bm q}\|_{\mathcal{T}_h}\|\varepsilon_h^y\|_{\mathcal{T}_h}\lesssim h^{k+1}(|\bm q|_{k+1}+|y|_{k+1})\| \varepsilon_h^y \|_{\mathcal{T}_h},\\
		T_5 &\lesssim h^\frac{1}{2} \| \varepsilon_h^y-\varepsilon_h^{\widehat{y}} \|_{\partial \mathcal{T}_h} \| \varepsilon_h^y \|_{\mathcal{T}_h}\lesssim h^{k+1}(|\bm q|_{k+1}+|y|_{k+1})\| \varepsilon_h^y \|_{\mathcal{T}_h}.
	\end{align*}
	Summing $T_1$ to $T_5$ gives
	\begin{align*}
	\| \varepsilon_h^y \|_{\mathcal{T}_h} \lesssim h^{k+1} (|\bm q|_{k+1}+|y|_{k+1}).
	\end{align*}
\end{proof}

The triangle inequality gives convergence rates for $\|\bm q -\bm q_h(u)\|_{\mathcal T_h}$ and $\|y -y_h(u)\|_{\mathcal T_h}$:

\begin{lemma}\label{le}
	\begin{subequations}
		\begin{align}
		\|\bm q -\bm q_h(u)\|_{\mathcal T_h} &\le \|\delta^{\bm q}\|_{\mathcal T_h} + \|\varepsilon_h^{\bm q}\|_{\mathcal T_h} \lesssim h^{k}(|\bm q|_{k+1}+|y|_{k+1}),\\
		\|y -y_h(u)\|_{\mathcal T_h} &\le \|\delta^{y}\|_{\mathcal T_h} + \|\varepsilon_h^{y}\|_{\mathcal T_h}  \lesssim  h^{k+1}(|\bm q|_{k+1}+|y|_{k+1}).
		\end{align}
	\end{subequations}
\end{lemma}

\subsubsection{Step 4: The error equation for part 2 of the auxiliary problem \eqref{EDG_u_b}.} \label{subsec:proof_step_4}

Next, we consider the dual equation \eqref{mixed_c}-\eqref{mixed_d} in the optimality system and compare with the second part of the auxiliary EDG equation \eqref{EDG_u_b}.  We split the errors as before; define
\begin{align}\label{notation_1}
\begin{aligned}%column 1
\delta^{\bm p} &=\bm p-{\bm\Pi}_V\bm p,\\
\delta^z&=z- \Pi_W z,\\
\delta^{\widehat z} &= z-I_h z,\\
\widehat {\bm\delta}_2 &= \delta^{\bm p}\cdot\bm n+  h^{-1} (\delta^z- \delta^{\widehat{z}}),
\end{aligned}
&&
\begin{aligned}%column 2
\varepsilon^{\bm p}_h &= {\bm\Pi}_V \bm p-\bm p_h(u),\\
\varepsilon^{z}_h &= \Pi_W z-z_h(u),\\
\varepsilon^{\widehat z}_h &= I_h z-\widehat{z}_h(u),\\
\widehat {\bm \varepsilon }_2 &= \varepsilon_h^{\bm p}\cdot\bm n+h^{-1} (\varepsilon^z_h-\varepsilon_h^{\widehat z}),
\end{aligned}
\end{align}
where $\widehat z_h(u) = \widehat z_h^o(u)$ on $\varepsilon_h^o$ and $\widehat z_h(u) = 0$ on $\varepsilon_h^{\partial}$.  This gives $\varepsilon_h^{\widehat z} = 0$ on $\varepsilon_h^{\partial}$.

\begin{lemma}\label{lemma:step2_first_lemma}
	We have
	\begin{align}\label{error_z}
	\hspace{3em}&\hspace{-3em} \mathscr B(\varepsilon^{\bm p}_h,\varepsilon^z_h,\varepsilon^{\widehat{z}}_h;\bm r_2, w_2,\mu_2) \ \nonumber\\
	&=\langle\delta^{\widehat{z}},\bm r_2\cdot \bm n\rangle_{\partial \mathcal T_h}+(y_h(u)- y, w_2)_{\mathcal T_h}+  \langle \widehat{\bm \delta}_2,w_2 - \mu_2\rangle_{\partial\mathcal T_h}.
	\end{align}
\end{lemma}
The proof is similar to the proof of  \Cref{lemma:step2_first_lemma} and is omitted.

\subsubsection{Step 5: Estimates for $\varepsilon_h^p$ and $\varepsilon_h^z$ by an energy and duality argument.} \label{subsec:proof_step_5}

\begin{lemma}\label{e_sec}
	Let $ \kappa$ be any positive constant.  Then there exists a constant $ C $ that does not depend on $ \kappa $ such that%  We have
	\begin{align}
	\|\varepsilon_h^{\bm p}\|_{\mathcal T_h}&+h^{-\frac{1}{2}}\|\varepsilon_h^z-\varepsilon_h^{\widehat z}\|_{\partial\mathcal T_h} \le  \mathbb E + \kappa  \| \varepsilon^z_h \|_{\mathcal T_h},\label{error_es_p}
	\end{align}
	where
	\begin{align*}
	\mathbb E  =   Ch^{-\frac{1}{2}}\|  \delta^{\widehat{z}}  \|_{\partial\mathcal T_h} + Ch^{-\frac{1}{2}}\|  \delta^{z}  \|_{\partial\mathcal T_h} + \frac  C {\kappa} \| y_h(u) - y \|_{\mathcal T_h} +C\| \delta^{\bm p} \|_{ \mathcal{T}_h}.
	\end{align*}
	%and $ C $ does not depend on $ \kappa $.
\end{lemma}

\begin{proof}
	Taking $(\bm r_2,w_2,\mu_2) = (\varepsilon^{\bm p}_h,\varepsilon^z_h,\varepsilon^{\widehat z}_h)$ in \eqref{error_z} in  \Cref{lemma:step2_first_lemma} gives
	\begin{align*}
	\hspace{1em}&\hspace{-1em}  \mathscr B ( \varepsilon^{ \bm p}_h,  \varepsilon^z_h, \varepsilon^{\widehat z}_h;\varepsilon^{\bm p}_h, \varepsilon^z_h, \varepsilon^{\widehat z}_h )\\
	& =\langle\delta^{\widehat{z}},\varepsilon_h^{\bm p}\cdot \bm n\rangle_{\partial \mathcal T_h}+(y_h(u)- y, \varepsilon_h^z)_{\mathcal T_h}+  \langle \widehat{\bm \delta}_2,\varepsilon_h^z - \varepsilon_h^{\widehat{z}}\rangle_{\partial\mathcal T_h}\\
	&\le Ch^{-\frac{1}{2}}\|\delta^{\widehat{z}}\|_{\partial \mathcal{T}_h}\| \varepsilon_h^{\bm p} \|_{\mathcal{T}_h}+\frac{1}{\kappa}\| y_h(u)-y \|_{\mathcal{T}_h}^2+\kappa \| \varepsilon_h^z \|_{\mathcal{T}_h}^2 \\
	&\quad + C(h^{\frac{1}{2}}\| \delta^{\bm p} \|_{\mathcal{T}_h} +h^{-\frac{1}{2}}\| \delta^z-\delta^{\widehat{z}} \|_{\partial \mathcal{T}_h}) h^{-\frac{1}{2}}\| \varepsilon_h^z-\varepsilon_h^{\widehat{z}} \|_{\partial \mathcal{T}_h}.
	\end{align*}
	\Cref{property_B} gives
	\begin{align*}
	\|\varepsilon_h^{\bm p}\|_{\mathcal T_h}&+h^{-\frac{1}{2}}\|\varepsilon_h^z-\varepsilon_h^{\widehat z}\|_{\partial\mathcal T_h}\nonumber\\
	&\le  Ch^{-\frac{1}{2}}\|  \delta^{\widehat{z}}  \|_{\partial\mathcal T_h} + Ch^{-\frac{1}{2}}\|  \delta^{z}  \|_{\partial\mathcal T_h} + \frac  C {\kappa} \| y_h(u) - y \|_{\mathcal T_h} \\
	&\quad+C\| \delta^{\bm p} \|_{ \mathcal{T}_h}+ \kappa  \| \varepsilon^z_h \|_{\mathcal T_h},
	\end{align*}
	where $\kappa$ is any positive constant.
\end{proof}

\begin{lemma}
	We have
	\begin{subequations}
		\begin{align}
		\norm{\varepsilon_h^{\bm p}}_{\mathcal T_h}	&\lesssim h^{k}(|\bm q|_{k+1}+|y|_{k+1}+|\bm p|_{k+1}+|z|_{k+1}),\label{var_p}\\
		\|\varepsilon^z_h\|_{\mathcal T_h} &\lesssim h^{k+1}(|\bm q|_{k+1}+|y|_{k+1}+|\bm p|_{k+1}+|z|_{k+1}).\label{var_z}
		\end{align}
	\end{subequations}
\end{lemma}

\begin{proof}
	First, take  $(\bm r_2,w_2,\mu_2) = ( {\bm\Pi}_V\bm{\Phi},-{\Pi}_W\Psi,-I_h\Psi)$ in \eqref{error_z} in  \Cref{lemma:step2_first_lemma} to obtain
	\begin{align*}
	\mathscr B &(\varepsilon^{\bm p}_h,\varepsilon^z_h,\varepsilon^{\widehat z}_h;{\bm\Pi}_V\bm{\Phi},-{\Pi}_W\Psi,-I_h\Psi)\\
	& = \langle\delta^{\widehat{z}},\bm \Pi_V \bm \Phi\cdot \bm n\rangle_{\partial \mathcal T_h}-(y_h(u)- y, \Pi_W \Psi)_{\mathcal T_h}-  \langle \widehat{\bm \delta}_2,\delta^\Psi - \delta^{\widehat{\Psi}}\rangle_{\partial\mathcal T_h}.
	\end{align*}
	Next, consider the dual problem \eqref{Dual_PDE} and let $\Theta = -\varepsilon_h^z$.  Using the definition of $\mathscr B$ and the proof technique for \Cref{dual_y} gives
	\begin{align*}
	\hspace{1em}&\hspace{-1em}  \mathscr B (\varepsilon^{\bm p}_h,\varepsilon^z_h,\varepsilon^{\widehat z}_h;{\bm\Pi}_V\bm{\Phi},-{\Pi}_W\Psi,-I_h\Psi)\\
	&=(\varepsilon_h^z,\varepsilon_h^z)_{\mathcal{T}_h}+\langle \varepsilon_h^z-\varepsilon_h^{\widehat{z}},\delta^{\bm \Phi}\cdot \bm n \rangle_{\partial \mathcal{T}_h}-\langle \varepsilon_h^{\bm p}\cdot\bm n,\delta^{\widehat{\Psi}} \rangle_{\partial \mathcal{T}_h}\\
	&\quad +\langle h^{-1} (\varepsilon_h^z-\varepsilon_h^{\widehat{z}}),\delta^\Psi-\delta^{\widehat{\Psi}} \rangle_{\partial \mathcal{T}_h}.
	\end{align*}
	Here, we used $\langle\varepsilon^{\widehat z}_h,\bm \Phi\cdot\bm n\rangle_{\partial\mathcal T_h}=0$, which holds since  $\varepsilon^{\widehat z}_h$ is a single-valued function on interior edges and $\varepsilon^{\widehat z}_h=0$ on $\varepsilon^{\partial}_h$. 
	
	Comparing the above two equalities gives
	\begin{align*}
	\|  \varepsilon_h^z\|_{\mathcal T_h}^2 & = \langle\delta^{\widehat{z}},\bm \Pi_V \bm \Phi\cdot \bm n\rangle_{\partial \mathcal T_h}-(y_h(u)- y, \Pi_W \Psi)_{\mathcal T_h}-  \langle \widehat{\bm \delta}_2,\delta^\Psi - \delta^{\widehat{\Psi}}\rangle_{\partial\mathcal T_h}\\
	&\quad-\langle \varepsilon_h^z-\varepsilon_h^{\widehat{z}},\delta^{\bm \Phi}\cdot \bm n \rangle_{\partial \mathcal{T}_h}+\langle \varepsilon_h^{\bm p}\cdot\bm n,\delta^{\widehat{\Psi}} \rangle_{\partial \mathcal{T}_h}\\
	&\quad -\langle h^{-1} (\varepsilon_h^z-\varepsilon_h^{\widehat{z}}),\delta^\Psi-\delta^{\widehat{\Psi}} \rangle_{\partial \mathcal{T}_h},\\
	&=:  \sum_{i=1}^6 R_i.
	\end{align*}
	Let $ C_0 = \max\{C, 1\} $, where $ C $ is the constant defined in  \eqref{classical_ine}. 
	For the terms $R_1$, $R_2$, and $R_3$, we have
	\begin{align*}
		R_1&=-\langle \delta^{\widehat{z}}, \delta^{\bm \Phi}\cdot \bm n\rangle_{\partial \mathcal{T}_h} \le C_0 h^{\frac{1}{2}}\| \delta^{\widehat{z}} \|_{\partial \mathcal{T}_h} \| \bm \Phi \|_{1,\Omega}\le C_0 C_{\text{reg}} h^\frac{1}{2} \| \delta^{\widehat{z}} \|_{\partial \mathcal{T}_h} \| \varepsilon_h^z \|_{\mathcal{T}_h},\\
		R_2&\le \| y_h(u)-y \|_{\mathcal{T}_h} (\| \delta^{\Psi} \|_{\mathcal{T}_h}+\| \Psi \|_{\Omega})\le C_0C_{\text{reg}} \| y-y_h(u) \|_{\mathcal{T}_h}\| \varepsilon_h^z \|_{\mathcal{T}_h},\\
		R_3&\le C_0 h^{\frac{3}{2}} (\delta^{\bm p}\cdot\bm n+\frac{1}{h}\|\delta^z-\delta^{\widehat{z}}\|_{\partial \mathcal{T}_h})\| \Psi \|_{2,\Omega}\\
		&\le C_0 C_{\text{reg}}(h\| \delta^{\bm p} \|_{\mathcal{T}_h}+h^\frac{1}{2} \|\delta^z-\delta^{\widehat{z}}\|_{\partial \mathcal{T}_h}) \| \varepsilon_h^z \|_{\mathcal{T}_h}.
	\end{align*}
	For the terms $R_4$, $R_5$ and $R_6$,  \Cref{e_sec} gives
	\begin{align*}
	%%%%%%%%%%%%%%%%%%%%%%%%%%%%%%
	R_4&= C_0 h^{\frac{1}{2}}\| \varepsilon_h^z-\varepsilon_h^{\widehat{z}} \|_{\partial \mathcal{T}_h} \| \bm \Phi \|_{1,\Omega}\le C_0C_{\text{reg}}h(\mathbb{E}+\kappa \|\varepsilon_h^z\|_{\mathcal{T}_h})\| \varepsilon_h^z \|_{\mathcal{T}_h},  \\
	R_5 &\le C_0 h^{\frac{3}{2}}\| \varepsilon_h^{\bm p} \|_{\partial \mathcal{T}_h}\|\Psi\|_{2,\Omega}\le C_0C_{\text{reg}} h\|\varepsilon_h^{\bm p}\|_{\mathcal{T}_h}\|\varepsilon_h^z\|_{\mathcal{T}_h}\\
	&\le C_0C_{\text{reg}} h(\mathbb{E}+\kappa\| \varepsilon_h^z \|_{\mathcal{T}_h})\|\varepsilon_h^z\|_{\mathcal{T}_h}, \\
	R_6 &\le C_0 h^{\frac{1}{2}}\| \varepsilon_h^z-\varepsilon_h^{\widehat{z}} \|_{\partial \mathcal{T}_h}\| \Psi \|_{2,\Omega} \le C_0C_{\text{reg}} h^{\frac{1}{2}}(\mathbb{E}+\kappa \| \varepsilon_h^z \|_{\mathcal{T}_h})\|\varepsilon_h^z\|_{\mathcal{T}_h}.
	\end{align*}
	
	Summing $R_1$ to $R_6$ gives
	\begin{align*}
	\|\varepsilon^z_h\|_{\mathcal T_h} &\le  C ( h\|\delta^{\bm p}\|_{\mathcal{T}_h} +\norm{y - y_h(u)}_{\mathcal T_h} + h^{1/2}\|{\delta^{\widehat z}}\|_{\partial\mathcal T_h}+h^{1/2}\|{\delta^{z}}\|_{\partial\mathcal T_h})\\
	&\quad+\mathbb C(\mathbb E +\kappa\|\varepsilon^z_h\|_{\mathcal T_h} ),
	\end{align*}
	where $\mathbb C =3 C_0C_{\text{reg}}$. Choose $\kappa=\frac{1}{2\mathbb C}$ gives
	\begin{align*}
	\|\varepsilon^z_h\|_{\mathcal T_h}\lesssim h^{k+1}(|\bm q|_{k+1}+|y|_{k+1}+|\bm p|_{k+1}+|z|_{k+1}).
	\end{align*}
	Finally, \eqref{error_es_p} and \eqref{var_z} imply  \eqref{var_p}.
\end{proof}

The triangle inequality again gives convergence rates for $\|\bm p -\bm p_h(u)\|_{\mathcal T_h}$ and $\|z -z_h(u)\|_{\mathcal T_h}$:
\begin{lemma}\label{lemma:step3_conv_rates}
	\begin{subequations}
		\begin{align}
		\|\bm p -\bm p_h(u)\|_{\mathcal T_h} &\le \|\delta^{\bm p}\|_{\mathcal T_h} + \|\varepsilon_h^{\bm p}\|_{\mathcal T_h} \ \nonumber \\ 
		&\lesssim h^{k}(|\bm q|_{k+1}+|y|_{k+1}+|\bm p|_{k+1}+|z|_{k+1}),\\
		\|z -z_h(u)\|_{\mathcal T_h} &\le \|\delta^{z}\|_{\mathcal T_h} + \|\varepsilon_h^{z}\|_{\mathcal T_h} \ \nonumber \\
		& \lesssim  h^{k+1}(|\bm q|_{k+1}+|y|_{k+1}+|\bm p|_{k+1}+|z|_{k+1}).
		\end{align}
	\end{subequations}
\end{lemma}

\subsubsection{Step 6: Estimates for $\|u-u_h\|_{\mathcal T_h}$, $\norm {y-y_h}_{\mathcal T_h}$, and $\norm {z-z_h}_{\mathcal T_h}$.}

Next, we compare the auxiliary problem to the EDG discretization of the optimality system \eqref{EDG_full_discrete}.  The resulting error bounds along with the earlier error bounds in  \Cref{le} and \Cref{lemma:step3_conv_rates} give the main convergence result.

The proofs in the final steps are similar to the HDG work \cite{HuShenSinglerZhangZheng_HDG_Dirichlet_control2}; we include the proofs here for completeness.

For the remainder of the proof, let
\begin{equation*}
\begin{split}
\zeta_{\bm q} &=\bm q_h(u)-\bm q_h,\quad\zeta_{y} = y_h(u)-y_h,\quad\zeta_{\widehat y} = \widehat y_h(u)-\widehat y_h,\\
\zeta_{\bm p} &=\bm p_h(u)-\bm p_h,\quad\zeta_{z} = z_h(u)-z_h,\quad\zeta_{\widehat z} = \widehat z_h(u)-\widehat z_h.
\end{split}
\end{equation*}
Subtracting the auxiliary problem and the EDG problem yields the error equations
\begin{subequations}\label{eq_yh}
	\begin{align}
	\mathscr B(\zeta_{\bm q},\zeta_y,\zeta_{\widehat y};\bm r_1, w_1,\mu_1)&=(u-u_h,w_1)_{\mathcal T_h}\label{eq_yh_yhu},\\
	\mathscr B(\zeta_{\bm p},\zeta_z,\zeta_{\widehat z};\bm r_2, w_2,\mu_2)&=-(\zeta_y, w_2)_{\mathcal T_h}\label{eq_zh_zhu}.
	\end{align}
\end{subequations}
\begin{lemma}
	We have
	\begin{align}\label{eq_uuh_yhuyh}
	\hspace{3em}&\hspace{-3em}  \gamma\|u-u_h\|^2_{\mathcal T_h}+\|y_h(u)-y_h\|^2_{\mathcal T_h}\nonumber\\
	&=( z_h-\gamma u_h,u-u_h)_{\mathcal T_h}-(z_h(u)-\gamma u,u-u_h)_{\mathcal T_h}.
	\end{align}
\end{lemma}
\begin{proof}
	First,
	\begin{align*}
	\hspace{3em}&\hspace{-3em}  ( z_h-\gamma u_h,u-u_h)_{\mathcal T_h}-( z_h(u)-\gamma u,u-u_h)_{\mathcal T_h}\\
	&=-(\zeta_{ z},u-u_h)_{\mathcal T_h}+\gamma\|u-u_h\|^2_{\mathcal T_h}.
	\end{align*}
	Next, \Cref{identical_equa} and \eqref{eq_yh} give
	\begin{align*}
	0 &= \mathscr B (\zeta_{\bm q},\zeta_y,\zeta_{\widehat{y}};\zeta_{\bm p},-\zeta_{z},-\zeta_{\widehat z}) + \mathscr B(\zeta_{\bm p},\zeta_z,\zeta_{\widehat z};-\zeta_{\bm q},\zeta_y,\zeta_{\widehat{y}})\\
	  &= - ( u- u_h,\zeta_{ z})_{\mathcal{T}_h}-\|\zeta_{ y}\|^2_{\mathcal{T}_h}.
	\end{align*}
	This gives $ -(u-u_h,\zeta_{ z})_{\mathcal{T}_h}=\|\zeta_{ y}\|^2_{\mathcal{T}_h} $, which completes the proof.
\end{proof}

\begin{theorem}\label{thm:estimates_u_y_z}
	We have
	\begin{subequations}
		\begin{align}\label{err_yhu_yh}
		\|u-u_h\|_{\mathcal T_h}&\lesssim h^{k+1}(|\bm q|_{k+1}+|y|_{k+1}+|\bm p|_{k+1}+|z|_{k+1}),\\
		\|y-y_h\|_{\mathcal T_h}&\lesssim h^{k+1}(|\bm q|_{k+1}+|y|_{k+1}+|\bm p|_{k+1}+|z|_{k+1}),\\
		\|z-z_h\|_{\mathcal T_h}&\lesssim h^{k+1}(|\bm q|_{k+1}+|y|_{k+1}+|\bm p|_{k+1}+|z|_{k+1}).\label{estimate_z}
		\end{align}
	\end{subequations}
\end{theorem}
\begin{proof}
	As mentioned earlier, the exact and approximate optimal controls satisfy $ \gamma u = z $ and $ \gamma u_h = z_h $; see \eqref{eq_adeq_e} and \eqref{EDG_full_discrete_e}.  Using these equations with the lemma above give
	\begin{align*}
	\hspace{3em}&\hspace{-3em}   \gamma\|u-u_h\|^2_{\mathcal T_h}+\|\zeta_{ y}\|^2_{\mathcal T_h}\\
	&=( z_h-\gamma u_h,u-u_h)_{\mathcal T_h}-( z_h(u)-\gamma u,u-u_h)_{\mathcal T_h}\\
	&=-( z_h(u)- z,u-u_h)_{\mathcal T_h}\\
	&\le \| z_h(u)- z\|_{\mathcal T_h} \|u-u_h\|_{\mathcal T_h}\\
	&\le\frac{1}{2\gamma}\| z_h(u)- z\|^2_{\mathcal T_h}+\frac{\gamma}{2}\|u-u_h\|^2_{\mathcal T_h}.
	\end{align*}
	\Cref{lemma:step3_conv_rates} gives
	\begin{align}\label{eqn:estimate_u_zeta_y}
	\|u-u_h\|_{\mathcal T_h}+\|\zeta_{ y}\|_{\mathcal T_h}&\lesssim h^{k+1}(|\bm q|_{k+1}+|y|_{k+1}+|\bm p|_{k+1}+|z|_{k+1}).
	\end{align}
	Use the triangle inequality and  \Cref{le} to obtain
	\begin{align*}
	\|y-y_h\|_{\mathcal T_h}&\lesssim h^{k+1}(|\bm q|_{k+1}+|y|_{k+1}+|\bm p|_{k+1}+|z|_{k+1}).
	\end{align*}
	Finally, the above estimate \eqref{eqn:estimate_u_zeta_y} for $ u $ along with $z = \gamma u $ and $z_h = \gamma u_h$ give the estimate \eqref{estimate_z} for $ z $.
%	\begin{align*}
%	\|z-z_h\|_{\mathcal T_h}&\lesssim h^{k+1}(|\bm q|_{k+1}+|y|_{k+1}+|\bm p|_{k+1}+|z|_{k+1}).
%	\end{align*}
\end{proof}

\subsubsection{Step 7: Estimates for $\|q-q_h\|_{\mathcal T_h}$ and  $\|p-p_h\|_{\mathcal T_h}$.}
\begin{lemma}
	We have
	\begin{subequations}
		\begin{align}
		\|\zeta_{\bm q}\|_{\mathcal T_h} &\lesssim h^{k+1}(|\bm q|_{k+1}+|y|_{k+1}+|\bm p|_{k+1}+|z|_{k+1}),\label{err_Lhu_qh}\\
		\|\zeta_{\bm p}\|_{\mathcal T_h} &\lesssim h^{k+1}(|\bm q|_{k+1}+|y|_{k+1}+|\bm p|_{k+1}+|z|_{k+1}).\label{err_Lhu_ph}
		\end{align}
	\end{subequations}
\end{lemma}
\begin{proof}
	\Cref{property_B}, the error equation \eqref{eq_yh_yhu}, and the estimate \eqref{eqn:estimate_u_zeta_y} give
	\begin{align*}
	\|\zeta_{\bm q}\|^2_{\mathcal T_h} &\lesssim  \mathscr B(\zeta_{\bm q},\zeta_y,\zeta_{\widehat y};\zeta_{\bm q},\zeta_y,\zeta_{\widehat y})\\
	&=( u- u_h,\zeta_{ y})_{\mathcal T_h}\\
	&\le\| u- u_h\|_{\mathcal T_h}\|\zeta_{ y}\|_{\mathcal T_h}\\
	&\lesssim h^{2k+2}(|\bm q|_{k+1}+|y|_{k+1}+|\bm p|_{k+1}+|z|_{k+1})^2.
	\end{align*}
	Similarly, \Cref{property_B}, the error equation \eqref{eq_zh_zhu}, the estimate \eqref{eqn:estimate_u_zeta_y},  \Cref{lemma:step3_conv_rates}, and  \Cref{thm:estimates_u_y_z} give
	\begin{align*}
	\|\zeta_{\bm p}\|^2_{\mathcal T_h} &\lesssim  \mathscr B(\zeta_{\bm p},\zeta_z,\zeta_{\widehat z};\zeta_{\bm p},\zeta_z,\zeta_{\widehat z})\\
	&=-(\zeta_{ y},\zeta_{ z})_{\mathcal T_h}\\
	&\le\|\zeta_{y}\|_{\mathcal T_h}\|\zeta_{ z}\|_{\mathcal T_h}\\
	&\le\|\zeta_{y}\|_{\mathcal T_h} ( \| z_h(u) - z \|_{\mathcal T_h} + \| z - z_h \|_{\mathcal T_h} )\\
	&\lesssim h^{2k+2}(|\bm q|_{k+1}+|y|_{k+1}+|\bm p|_{k+1}+|z|_{k+1})^2.
	\end{align*}
\end{proof}
The above lemma, the triangle inequality,  \Cref{le}, and  \Cref{lemma:step3_conv_rates} complete the proof of the main result:
\begin{theorem}
	We have
	\begin{subequations}
		\begin{align}
		\|\bm q-\bm q_h\|_{\mathcal T_h}&\lesssim h^{k}(|\bm q|_{k+1}+|y|_{k+1}+|\bm p|_{k+1}+|z|_{k+1}),\label{err_q}\\
		\|\bm p-\bm p_h\|_{\mathcal T_h}&\lesssim h^{k}(|\bm q|_{k+1}+|y|_{k+1}+|\bm p|_{k+1}+|z|_{k+1})\label{err_p}.
		\end{align}
	\end{subequations}
\end{theorem}

\section{Numerical Experiments}
\label{sec:numerics}

Next, we present a numerical example to illustrate our theoretical results.  We consider the distributed control problem for the Poisson equation on a square domain $\Omega = [0,1]\times [0,1] \subset \mathbb{R}^2$ and take $\gamma = 1$.  We set the exact state and dual state to be $ y(x_1,x_2) = \sin(\pi x_1) $ and $ z(x_1,x_2) = \sin(\pi x_1)\sin(\pi x_2)$, and generate the data $f$, $ g $, and $y_d$ from the optimality system \eqref{eq_adeq}.  Numerical results for $ k = 1 $ and $ k = 2 $ for this problem are shown in Table \ref{table_1}--Table \ref{table_2}.  The numerical convergence rates match the theory.

	\begin{table}%[!hbp]
		\begin{center}
			\begin{tabular}{|c|c|c|c|c|c|}
				\hline
				$h/\sqrt 2$ &$1/8$& $1/16$&$1/32$ &$1/64$ & $1/128$ \\
				\hline
				$\norm{\bm{q}-\bm{q}_h}_{0,\Omega}$&3.6714e-01   &1.8490e-01   &9.2615e-02   &4.6328e-02   &2.3167e-02 \\
				\hline
				order&-& 0.99& 1.00  &1.00& 1.00\\
				\hline
				$\norm{\bm{p}-\bm{p}_h}_{0,\Omega}$& 3.8422e-01   &1.9228e-01   &9.6161e-02   &4.8083e-02   &2.4042e-02 \\
				\hline
				order&-&  1.00&1.00 &1.00 & 1.00 \\
				\hline
				$\norm{{y}-{y}_h}_{0,\Omega}$&2.4802e-02   &6.3399e-03   &1.5989e-03   &4.0125e-04   &1.0049e-04\\
				\hline
				order&-& 1.97&1.99&2.00 & 2.00 \\
				\hline
				$\norm{{z}-{z}_h}_{0,\Omega}$& 2.8282e-02   &7.0802e-03   &1.7694e-03   &4.4218e-04   &1.1052e-04 \\
				\hline
				order&-& 2.00&2.00&2.00& 2.00 \\
				\hline
			\end{tabular}
		\end{center}
		\caption{ Errors for the state $y$, adjoint state $z$, and the fluxes $\bm q$ and $\bm p$ when $k=1$.}\label{table_1}
	\end{table}

	\begin{table}%[!hbp]
		\begin{center}
			\begin{tabular}{|c|c|c|c|c|c|}
				\hline
				$h/\sqrt 2$ &$1/8$& $1/16$&$1/32$ &$1/64$ & $1/128$ \\
				\hline
				$\norm{\bm{q}-\bm{q}_h}_{0,\Omega}$&2.6598e-02   &6.7755e-03   &1.7029e-03   &4.2631e-04   &1.0662e-04 \\
				\hline
				order&-& 1.97& 1.99  &2.00& 2.00\\
				\hline
				$\norm{\bm{p}-\bm{p}_h}_{0,\Omega}$& 2.6694e-02   &7.0861e-03   &1.7999e-03   &4.5178e-04   &1.1306e-04 \\
				\hline
				order&-&  1.91&1.98 &1.99 & 2.00 \\
				\hline
				$\norm{{y}-{y}_h}_{0,\Omega}$&8.3274e-04   &1.0672e-04   &1.3592e-05   &1.7164e-06   &2.1566e-07\\
				\hline
				order&-& 2.96&2.97&2.99 & 3.00 \\
				\hline
				$\norm{{z}-{z}_h}_{0,\Omega}$& 1.4515e-03   &1.9483e-04  & 2.5202e-05   &3.2009e-06   &4.0316e-07\\
				\hline
				order&-& 2.90&2.95&2.98& 2.99\\
				\hline
			\end{tabular}
		\end{center}
		\caption{Errors for the state $y$, adjoint state $z$, and the fluxes $\bm q$ and $\bm p$ when $k=2$.}\label{table_2}
	\end{table}

\section{Conclusions}
We proposed an EDG method to approximate the solution of an optimal distributed control problems for the Poisson equation. We obtained optimal a priori error estimates for the control, state, and dual state, but suboptimal estimates for their fluxes.  As mentioned earlier, EDG has potential for optimal control problems involving convection dominated partial differential equations and fluid flows.  These problems would be interesting to explore in the future.

Also, we used a different EDG error analysis strategy to prove the error estimates in this work.  We are currently investigating another EDG method, and we have used the different analysis approach to prove optimal convergence rates for all variables.  The details will be reported in a future paper.

\bibliographystyle{siamplain}
\bibliography{yangwen_ref_papers}
\end{document}